\documentclass{scrartcl}
\usepackage[
color, 
noindent, 
english, 
enumstyle, 
theorem, 
operator, 
alphabet, 
cmd, 
arraytweak,
numberbysection,
arraytweak
]
{jgrygierek}
\usepackage[draft=false]{scrlayer-scrpage} 
\usepackage[
backend=biber, 
maxnames=42,
style=alphabetic,
citestyle=alphabetic,
doi=false,
url=false,
isbn=false,
giveninits=true
]{biblatex} % because of 42 reasons ;-)
\patchcmd{\thebibliography}{\chapter*}{\section*}{}{}
\allowdisplaybreaks
\date{\today}
\author{\texorpdfstring{Jens Grygierek\footnotemark[1]}{Jens Grygierek}}
\title{Poisson fluctuations for edge counts in high-dimensional random geometric graphs} % no // to keep hyperref happy ;)
\makeatletter
\hypersetup{
  pdftitle    = {\@title},
  pdfsubject  = {Mathematics/Probability, 60D05, 60F05},
  pdfauthor   = {\@author},
  pdfkeywords = {\@title, \@author, Poisson limit theorem, edge counting statistic, high dimensional random geometric graph, Poisson point process, second-order Poincar\'{e} inequality, stochastic geometry, Mehler's formula, Stein's method, Malliavin calculus, phase transition}
}
\makeatother

\begin{document}
\renewcommand{\thefootnote}{\fnsymbol{footnote}}

\maketitle
	
\footnotetext[1]{Institute of Mathematics, Osnabr\"uck University, Germany. Email: jens.grygierek@uni-osnabrueck.de}

\begin{abstract}
	\noindent We prove a Poisson limit theorem in the total variation distance of functionals of a general Poisson point process using the Malliavin-Stein method.
	Our estimates only involve first and second order difference operators and are closely related to the corresponding bounds for the normal approximation in the Wasserstein distance by Last, Peccati and Schulte, see \cite{LastPeccatiSchulte2016}.
	As an application of this Poisson limit theorem, we consider a stationary Poisson point process in $\RR^d$ and connect any two points whenever their distance is less than or equal to a prescribed distance parameter.
	This construction gives rise to the well known random geometric graph.
	The number of edges of this graph is counted that have a midpoint in the $d$-dimensional unit ball.
	A quantitative Poisson limit theorem for this counting statistic is derived, as the space dimension $d$ and the intensity of the Poisson point process tend to infinity simultaneously, extending our previous work, \cite{GrygierekThaele2016} where we derived a central limit theorem, showing that the phase transition phenomenon holds also in the high-dimensional set-up.
	% http://www.ams.org/msc/msc2010.html
\end{abstract}

\section{Introduction and main results}

Fix an intensity $\lambda \in (0,\infty)$ and a distance parameter $\delta \in (0,\infty)$ and let $\eta_\lambda$ be a stationary Poisson point process in $\RR^d$, $d \in \NNN$ with intensity $\lambda$. The points of $\eta_\lambda$ are taken as the vertices of a random graph and we connect any two distinct vertices by an edge provided that their distance is less than or equal to $\delta$. By this construction the random geometric graph in $\RR^d$ arises.

This paper is a direct continuation of \cite{GrygierekThaele2016}, where we have derived a quantitative central limit theorem for the number of edges that have their midpoint in the $d$-dimensional unit ball $\BB^d$ as the space dimension $d$ and the intensity $\lambda$ tend to infinity simultaneously such that the expectation of the considered edge counting statistic tends to infinity. 
In this paper we derive the corresponding Poisson limit theorem in the case that the expectation tends to a positive but finite constant by first proving a general Poisson limit theorem for Poisson functionals using the Malliavin-Stein method, that comes in the taste of the remarkable central limit theorem \cite[Theorem 1.1]{LastPeccatiSchulte2016}.

\subsection{\texorpdfstring{Poisson approximation for $\NNN$-valued Poisson functionals}{Poisson approximation for N-valued Poisson functionals}}

We first rephrase a version of the main result from \cite{LastPeccatiSchulte2016}, a so-called second order Poincar\'{e} inequality for Poisson functionals, see also \cite[Theorem 2.13]{LastPenrose2018}, that only involves moments of first and second order difference operators.

\begin{theorem}\label{thm:gauss-limit}
	Let $F \in \dom{D}$ be a Poisson functional such that $\E*{F} = 0$ and $\V*{F} = 1$.
	Define
	\begin{align*}
		\gamma_1(F) & := \int\limits_{\XX^3} \! 
				\rbras*{\E*{(D_{x_1,x_3}^2F)^4}\E*{(D_{x_2,x_3}^2 F)^4} \E*{(D_{x_1}F)^4}\E*{(D_{x_2}F)^4}}^{\frac{1}{4}} 
				\mu^3(\id (x_1, x_2, x_3))\\
		\gamma_2(F) & := \int\limits_{\XX^3} \! 
				\rbras*{\E*{(D_{x_1,x_3}^2 F)^4}\E*{(D_{x_2,x_3}^2 F)^4}}^{\frac{1}{2}}
				\mu^3(\id (x_1, x_2, x_3))\\
		\gamma_{3,N}(F) & := \int\limits_{\XX}
				\Eabs*{D_xF}^3
				\mu(\id x)
	\end{align*}
	and let $Z$ be a standard Gaussian random variable,
	then
	\begin{align*}
		\dist_W(F,Z) \leq 2 \sqrt{\gamma_1(F)} + \sqrt{\gamma_2(F)} + \gamma_{3,N}(F),
	\end{align*}
	where $\dist_W$ denotes the Wasserstein-distance, see Definition \ref{def:wasserstein}.
\end{theorem}

Replacing the third term in the approximation bound with
\begin{align*}
	\gamma_{3,P}(F) & := \int\limits_{\XX} \rbras*{\Eabs*{D_xF(D_xF-1)}^2}^\frac{1}{2} \rbras*{\Eabs*{D_xF}^2}^\frac{1}{2} \mu(\id x),
\end{align*}
we can formulate the analogue of Theorem \ref{thm:gauss-limit} for Poisson approximation, which is our first main result and will be used later to derive Theorem \ref{thm:edge-count-poisson-limit}.

\begin{theorem}[Poisson Approximation]\label{thm:poisson-limit}
	Let $\eta$ be a Poisson point process on $\XX$ with $\sigma$-finite non-atomic intensity measure $\mu$ and let $F$ be an $\NNN$-valued Poisson functional satisfying $F \in \dom{D}$.
	Further, let $\cP(\theta)$ be a Poisson distributed random variable with parameter $\theta > 0$.
	Then
	\begin{align*}
		\dist_{TV}\rbras*{F,\cP(\theta)} \leq \frac{1-e^{-\theta}}{\theta} \rbras*{2\sqrt{\gamma_1(F)} + \sqrt{\gamma_2(F)} + \frac{\gamma_{3,P}(F)}{\theta} + \abs*{\E*{F} - \theta} + \abs*{\V*{F} - \theta}},
	\end{align*}
	where $\dist_{TV}$ denotes the total variation distance, see Definition \ref{def:total-variation}.
\end{theorem}

Note that $\gamma_1(F)$ and $\gamma_2(F)$ were also used before in the central limit theorem, which will be useful in the proof of our second main result, since it allows us to reuse some of the calculations we did in the previous work \cite{GrygierekThaele2016}.

\subsection{Poisson fluctuations for edge counts in high-dimensional random geometric graphs}

Let $\eta_d$ be a stationary Poisson point process on $\RR^d$ with dimension-dependent intensity $\lambda_d \in (0,\infty)$, i.e. the intensity measure is given by $\mu_d = \lambda_d \Lambda_d$, where $\Lambda_d$ denotes the $d$-dimensional Lebesgue measure.
We choose a dimension-dependent distance parameter $\delta_d$ with $\delta_d \rightarrow 0$ for $d \rightarrow \infty$, namely we take 
\begin{align*}
	\delta_d = \frac{1}{d},
\end{align*}
which implies that $\delta_d \in (0,1)$ for all $d \geq 2$. The motivation for our choice is explained in Remark \ref{rem:choice-of-delta} below, where we also give the precise conditions for $\delta_d$ to allow for more general choices. We notice that $\delta_d \rightarrow 0$. Finally we choose the dimension-dependent intensity $\lambda_d$ such that $\lambda_d \rightarrow \infty$ for $d \rightarrow \infty$.

Let $\cE(\lambda_d,\delta_d,d)$ denote the number of edges of the random geometric graph that have their midpoint in the $d$-dimensional unit ball $\BB^d$, that is the edge-counting statistic given by
\begin{align*}
	\cE(\lambda_d,\delta_d,d) := \frac{1}{2} \sum\limits_{(y_1,y_2) \in (\eta_{\lambda_d})^2_{\neq}} \1\cbras*{\norm{y_1 - y_2} \leq \delta_d, \frac{y_1+y_2}{2} \in \BB^d}.
\end{align*}
To simplify our notation we shall use the abbreviation $\cE_d$ for $\cE(\lambda_d,\delta_d,d)$. The expectation and the variance of $\cE_d$ was already derived in our previous work \cite[eq. 4, eq. 5, Lemma 7]{GrygierekThaele2016}, namely:
\begin{align*}
	\E*{\cE_d} = \frac{1}{2} \kappa_d^2 \lambda_d^2 \delta_d^d
\end{align*}
and
\begin{align*}
	\frac{1}{2} \kappa_d^2 \lambda_d^2 \delta_d^d + \rbras*{1 - \frac{\delta_d}{2}}^d \kappa_d^3 \lambda_d^3 \delta_d^{2d} \leq \V*{\cE_d} \leq \frac{1}{2} \kappa_d^2 \lambda_d^2 \delta_d^d + \rbras*{1 + \frac{\delta_d}{2}}^d \kappa_d^3 \lambda_d^3 \delta_d^{2d}.
\end{align*}

Here and below, $\kappa_d := \Lambda_d(\BB^d)$ denotes the volume of the $d$-dimensional unit ball. Note that the exponential decay of $\kappa_d$ behaves like $\frac{1}{\sqrt{\pi d}}\rbras*{\frac{2\pi e}{d}}^{\frac{d}{2}}$, as $d \rightarrow \infty$, according to Stirling's formula.

\medskip

We investigate the asymptotic distributional behavior of $\cE_d$ as $\delta_d \rightarrow 0$ and the intensity $\lambda_d$ as well as the space dimension $d$ tend to infinity simultaneously. This set-up is opposed to the most of the existing literature in which the focus lies on random geometric graphs in $\RR^d$ with some fixed space dimension $d$, see \cite{BubeckDingEldanRacz2016} and \cite{DevroyeGyorgyLugosiUdina2011} for notable exceptions, where, however, questions concerning the high-dimensional fluctuations are not touched.

The asymptotic behavior of $\cE_d$ depends on how fast the sequence $(\lambda_d)_{d \in \NNN}$ increases as $d \rightarrow \infty$.
This phenomenon is quite common for asymptotic results related to edge counts (or more generally subgraph counts) and component counts.
In particular, here, one has to distinguishes the following phases, determined by the limit of the expectation $\E*{\cE_d}$:

\begin{align}
	\label{eq:phase:1}
	\lim\limits_{d \rightarrow \infty} \frac{1}{2}\kappa_d^2 \lambda_d^2 \delta_d^d & = \infty,\\
	\label{eq:phase:2}
	\lim\limits_{d \rightarrow \infty} \frac{1}{2}\kappa_d^2 \lambda_d^2 \delta_d^d & = \theta \in (0,\infty),\\
	\label{eq:phase:3}
	\lim\limits_{d \rightarrow \infty} \frac{1}{2}\kappa_d^2 \lambda_d^2 \delta_d^d & = 0,
\end{align}

\begin{remark}
	If the expectation tends to infinity \eqref{eq:phase:1} the edge-counting statistic satisfies a central limit theorem, see \cite[Theorem 1]{GrygierekThaele2016}.
\end{remark}

In this paper, we obtain a Poisson limit theorem for a finite non-zero limit \eqref{eq:phase:2} showing that the phase-transition phenomenon for the edge-counting statistic holds also in the high-dimensional set-up:

\begin{theorem}\label{thm:edge-count-poisson-limit}
	Assume $\frac{1}{2}\kappa_d^2 \lambda_d^2 \delta_d^d \rightarrow \theta \in (0,\infty)$ for $d \rightarrow \infty$ and let $\cP(\theta)$ be a Poisson distributed random variable with parameter $\theta$.
	Then one can find absolute constants $\bC_1, \bC_2,\bD \in (0,\infty)$ such that
	\begin{align*}
		\dist_{TV}\rbras*{\cE_d,\cP(\theta)} \leq \bC_1 (\kappa_d \lambda_d \delta_d^d)^\frac{1}{2} + \bC_2 \abs*{\frac{1}{2}\kappa_d^2 \lambda_d^2 \delta_d^d - \theta},
	\end{align*}
	whenever $d \geq \bD$. In particular, one has that
	\begin{align*}
		\cE_d \os{D}{\longrightarrow} \cP(\theta), \quad \text{as} \quad d \rightarrow \infty.
	\end{align*}
\end{theorem}

\begin{remark}
	If the expectation tends to zero, \eqref{eq:phase:3}, we also have $\V*{\cE_d} \rightarrow 0$, indicating that the edge-counting statistic vanishes in the limit, since the random graph contains almost surely no edges.
\end{remark}

\bigskip

The rest of this text is structured as follows.
In Section \ref{sec:prelim}  we recall some necessary background material on Poisson functionals and the Malliavin-Stein method. In particular we introduce Mehler's formula that will be the core ingredient in the proof of Theorem \ref{thm:poisson-limit} in Section \ref{sec:proof:poisson-limit}.
In Section \ref{sec:general-bound} we derive a general bound for second order $U$-statistics. The final Section \ref{sec:proof:edge-count-poisson-limit} contains the proof of Theorem \ref{thm:edge-count-poisson-limit}.

\section{Preliminaries}\label{sec:prelim}

The $d$-dimensional Euclidean space is denoted by $\RR^d$ and we let $\sB^d$ be the Borel $\sigma$-field on $\RR^d$. The Lebesgue measure on $\RR^d$ is indicated by $\Lambda_d$.
A $d$-dimensional ball with radius $r > 0$ and center in $z \in \RR^d$ is defined by
\begin{align*}
	\BB^d_r(z) := \cbras*{x \in \RR^d : \norm{x-z} \leq r},
\end{align*}
where $\norm{\cdot}$ stands for the usual Euclidean norm.
We shall write $\BB^d$ instead of $\BB^d_1(0)$ and denote by
\begin{align*}
	\kappa_d := \Lambda_d(\BB^d) = \frac{\pi^\frac{d}{2}}{\Gamma\sbras*{1 + \frac{d}{2}}}
\end{align*}
the volume of the $d$-dimensional unit ball $\BB^d$, where $\Gamma\sbras*{\cdot}$ is Euler's gamma function.

We will use the Wasserstein-distance for the normal approximation and the total variation distance for the Poisson approximation, see for instance \cite[Section 2.1]{BourguinPeccati2016}.

\begin{definition}\label{def:wasserstein}
	We denote by $\Lip{1}$ the class of Lipschitz functions $h:\RR \rightarrow \RR$ with Lipschitz constant less or equal to one, i.e. $h$ is absolutely continuous and almost everywhere differentiable with $\norm{h'}_\infty \leq 1$.
	Given two $\RR$-valued random variables $X,Y$, with $\Eabs*{X} < \infty$ and $\Eabs*{Y} < \infty$ the Wasserstein distance between the laws of $X$ and $Y$, written $\dist_W(X,Y)$ is defined as
	\begin{align*}
		\dist_W(X,Y) := \sup\limits_{h \in \Lip{1}} \abs*{\E*{h(X)} - \E*{h(Y)}}.
	\end{align*}
\end{definition}

\begin{definition}\label{def:total-variation}
	Given two $\NNN$-valued random variables $X,Y$, the total variation distance between the laws of $X$ and $Y$, written $\dist_{TV}(X,Y)$ is defined as
	\begin{align*}
	\dist_{TV}(X,Y) := \sup\limits_{A \subseteq \NNN}\abs{\Prob*{X \in A} - \Prob*{Y \in A}}.
	\end{align*}
\end{definition}

\subsection{Poisson functionals and Malliavin-Stein Method}\label{sec:prelim:malliavin}

Let $(\XX, \sX, \mu)$ be a Borel measure space with $\sigma$-finite and non-atomic measure $\mu$ such that $\mu(\XX) > 0$. For $p > 0$ and $n \in \NNN$ we denote by $L^p(\mu^n)$ the set of all measurable functions $f: \XX^n \rightarrow \RR$ such that $\int\abs{f}^p \id \mu^n < \infty$.

We use the symbol $\rN_\sigma := \rN_\sigma(\XX)$ to indicate the class of all $\sigma$-finite measures $\chi$ on $\XX$ with $\chi(B) \in \NNN \cup \cbras{\infty}$ for all $B \in \sX$ and supply the space $\rN_\sigma$ with the smallest $\sigma$-field $\sN_\sigma := \sN_\sigma(\XX)$ such that all mappings of the form $\chi \mapsto \chi(B)$ with $\chi \in \rN$ and $B \in \sX$ are measurable.

It will be convenient for us to identify a counting measure $\chi \in \rN_\sigma$ with its support and to write $x \in \chi$ if the point $x \in \XX$ is charged by $\chi$.
The Dirac measure concentrated at a point $x \in \XX$ is denoted by $\delta_x$.
This construction mostly follows \cite{Peccati2012} and \cite{LastPeccatiSchulte2016}.
We let $(\Omega, \sF, \dsP)$ be our underlying probability space and denote by $L^p(\dsP)$, $p > 0$, the space of all random variables $Y:\Omega \rightarrow \RR$ such that $\Eabs{Y}^p < \infty$.

Consider a $\sigma$-finite non-atomic measure $\mu$ on $\XX$. 
A Poisson point process $\eta$ with intensity measure $\mu$ is a random counting measure on $\XX$, that is a random element in $\rN_\sigma$, such that
\begin{enumerate}
	\item For all $B \in \sX$ and all $k \in \NNN$ it holds, that $\eta(B) \overset{d}{\sim} \text{Po}_{\mu(B)}$, i.e.,
		\begin{align*}
			\Prob*{\eta(B) = k} = \frac{\mu(B)^k}{k!}e^{-\mu(B)},
		\end{align*}
		and for $\mu(B) = \infty$, we set $\frac{\infty^k}{k!}e^{-\infty} = 0$ for all $k$.
	\item For all $m \in \NN$ and all pairwise disjoint measurable sets $B_1, \ldots, B_m \in \sX$, the random variables $\eta(B_1), \ldots, \eta(B_m)$ are independent.
\end{enumerate}

By a Poisson functional $F$ we understand a random variable $F \in L^2(\dsP)$, that is almost surely of the form $F = f(\eta)$, where $f: \rN_\sigma \rightarrow \RR$ is some measurable function, the so-called representative of $F$.
For a Poisson functional $F$ with representative $f$ and $x \in \XX$ we define the first-order difference operator 
\begin{align}\label{eq:DxF}
	D_xF := D_xf(\eta) := f(\eta + \delta_x) - f(\eta),
\end{align}
and for $m \geq 2$ points $x_1, \ldots, x_m \in \XX$ the $m$-th-order difference operator $D_{x_1,\ldots, x_m}F$ is defined inductively by
\begin{align*}
	D^m_{x_1,\ldots,x_m}F & := D^{m-1}_{x_2, \ldots, x_m}(D_{x_1}F),
\end{align*}
where $D_{x}^1F = D_xF$.
Note that this definition does not depend on the choice of the representative $f$ $\mu^m$-a.e.\ and $\dsP$-a.s.\ and further that $D^m_{x_1,\ldots, x_m}F$ is symmetric in the arguments $x_1, \ldots, x_m$.

In the following we will denote by $DF$ resp. $D^mF$ the mappings
\begin{align*}
	\begin{array}{rrcl}
		DF:\XX \rightarrow \RR, & \quad x & \overset{DF}{\longmapsto} & D_xf(\eta),\\
		D^mF:\XX \rightarrow \RR, & \quad (x_1, \ldots, x_m) & \overset{D^mF}{\longmapsto} & D_{x_1, \ldots, x_m}^m f(\eta).
	\end{array}
\end{align*}

For a short introduction to the Malliavin-Calculus we recall some of the important tools in the development of the theory. For a deeper discussion of Fock Spaces and Chaos Expansion as well as Malliavin-Calculus and Malliavin-Stein Method we refer the reader to  \cite{Last2016} and the books \cite{LastPenrose2018,PeccatiReitzner2016}.
We introduce the notion of the Wiener-It\^{o} chaos expansion, see \cite{LastPeccatiSchulte2016} and the references therein, especially \cite{LastPenrose2011} for more details and proofs.

Every Poisson functional $F$ admits a representation of the type 
\begin{align*}
	F = \E*{F} + \sum\limits_{n = 1}^\infty I_n(f_n)
\end{align*}
where the series coverges in $L^2(\dsP)$. For each $n \geq 1$, the kernel $f_n$ is given by the (scaled) expectation of the $n$-order difference operator, i.e.\ $f_n := \frac{1}{n!} \E*{D^nF}$ and $I_n(\cdot)$ denotes the $n$-th order Wiener-It\^{o} integral. This representation is known as Wiener-It\^{o} chaos expansion of $F$.

We say a Poisson functional lies in the domain of $D$, $F \in \dom{D}$, if 
\begin{align*}
	\sum\limits_{n=1}^\infty n n! \norm{f_n}^2_{L^2(\mu^n)} < \infty.
\end{align*}
In this case $D$ is called the Malliavin derivative operator associated with the Poisson process $\eta$, and it holds $\dsP$-a.s. and $\mu$-a.e., $x \in \XX$, that
\begin{align*}
	D_xF = \sum_{n=1}^\infty n I_{n-1}(f_n(x,\cdot)),
\end{align*}
where the right hand side is the definition of the Malliavin derivative operator and the left hand side is the path-wise defined first-order difference operator given by \eqref{eq:DxF}.

Note that the following Lemma can be used to easily check if a Poisson functional lies in the domain of $D$.

\begin{lemma}[{\cite[Lemma 3.1]{PeccatiThaele2013}}]
	Let $F \in L^2(\dsP)$ denote a Poisson functional with representative $f$ such that
	\begin{align*}
		\dsE \int\limits_{\XX} (f(\eta+\delta_x) - f(\eta))^2 \mu(\id x) < \infty.
	\end{align*}
	Then $F \in \dom{D}$.
\end{lemma}

The Wiener It\^{o} chaos expansion gives rise to the Ornstein-Uhlenbeck generator $L$, that is defined for all Poisson functionals $F \in \dom{L}$, i.e. 
\begin{align*}
	\sum\limits_{n=1}^\infty n^2 n! \norm{f_n}^2_{L^2(\mu^n)} < \infty,
\end{align*}
by
\begin{align*}
	LF = - \sum\limits_{n=1}^\infty n I_n(f_n),
\end{align*}
and its (pseudo) inverse $L^{-1}$ is given by
\begin{align*}
	L^{-1}F := - \sum\limits_{n=1}^\infty \frac{1}{n} I_n(f_n).	
\end{align*}

In \cite[Section 3, Theorem 3.1]{PeccatiSoleTaqquUtzet2010} the Malliavin-Calculus was combined with Stein's method to derive a bound on the Wasserstein distance between the law of a standardized Poisson Functional $F \in \dom{D}$ and the standard Gaussian distribution. This bound as well as the bound derived in \cite[Theorem 3.1]{Peccati2012}, stated here as Theorem \ref{thm:poisson-limit:malliavin-bound} for Poisson approximation in the total variation distance rely on the inverse $L^{-1}$ of the Ornstein-Uhlenbeck generator $L$, which generally requires the calculation of the Wiener-It\^{o} chaos expansion of $F$.
In \cite{LastPeccatiSchulte2016} this was solved for the normal approximation case by establishing and applying a general Mehler formula for Poisson processes which allows to represent the inverse Ornstein-Uhlenbeck generator in terms of thinned Poisson point processes to derive bounds that only rely on the moments of the first- and second-order difference operators $D_xF$ and $D^2_{x_1,x_2}F$.

\subsection{Mehler's formula}\label{sec:prelim:mehler}

For the sake of brevity we only introduce Mehler's formula and the derived results we will need in the proof of Theorem \ref{thm:poisson-limit} and refer the reader for the full coverage to \cite{LastPeccatiSchulte2016}.

Let $s \in [0,1]$ and denote by $\eta^{(s)}$ the $s$-thinning of our Poisson point process $\eta$ and by $\Pi_\nu$ the distribution of a Poisson point process with intensity measure $\nu$. 
We define the operator $P_s$ by
\begin{align*}
	P_sF := \int \Econd*{f(\eta^{(s)} + \chi)}{\eta} \Pi_{(1-s)\mu}(\id \chi),
\end{align*}
where the conditional expectation is taken with respect to the random thinning and the Poisson point process $\chi$, conditioned on $\eta$. Using the operator $P_s$, we derive Mehler's formula:

\begin{theorem}[{Mehler's formula, \cite[Theorem 3.2]{LastPeccatiSchulte2016}}]\label{thm:mehlers-formula}
	Let $F$ be a Poisson functional and $\E*{F} = 0$, then we have $\dsP$-a.s. that
	\begin{align*}
		L^{-1}F = - \int\limits_0^1 s^{-1} P_sF \id s.
	\end{align*}
\end{theorem}

We will need the following inequalities in the proof of our Poisson limit theorem, Theorem \ref{thm:poisson-limit}.

\begin{lemma}[{\cite[Lemma 3.4]{LastPeccatiSchulte2016}}]\label{lem:malliavin:moments}
	Let $F$ be a Poisson functional and $p \geq 1$, then
	\begin{align*}
		\Eabs*{D_xL^{-1}F}^p \leq \Eabs*{D_xF}^p, \quad \mu\text{-a.e. } x \in \XX,
	\end{align*}
	and
	\begin{align*}
		\Eabs*{D^2_{x_1,x_2}L^{-1}F}^p \leq \Eabs*{D^2_{x_1,x_2}F}^p, \quad \mu^2\text{-a.e. } (x_1,x_2) \in \XX^2.
	\end{align*}
\end{lemma}

Since $L^{-1}F = L^{-1}(F - \E*{F})$ and $D_xF = D_x(F - \E*{F})$, we can rephrase the result on the covariance, see \cite[Theorem 4.1]{LastPeccatiSchulte2016}, to obtain a result on the variance of our Poisson functional $F$:

\begin{theorem}\label{thm:malliavin:var}
	Let $F \in \dom{D}$, then
	\begin{align*}
		\E*{\rbras*{\V*{F} - \int\limits_{\XX} (D_xF)(-D_xL^{-1}F) \mu(\id x)}^2} \leq 4 \gamma_1(F) + \gamma_2(F).
	\end{align*}
\end{theorem}

\section{Proof of Theorem \ref{thm:poisson-limit}}\label{sec:proof:poisson-limit}

Let us first recall the Malliavin bounds for Poisson approximation from \cite[Theorem 3.1]{Peccati2012}:

\begin{theorem}\label{thm:poisson-limit:malliavin-bound}
	Let $\eta$ be a Poisson point process on $\XX$ with $\sigma$-finite and non-atomic intensity measure $\mu$ and let $F$ be an $\NNN$-valued Poisson functional satisfying $F \in \dom{D}$. Further let $\cP(\theta)$ be a Poisson distributed random variable with parameter $\theta > 0$. Then
	\begin{align*}
		\dist_{TV}\rbras*{F,\cP(\theta)} & \leq  \frac{1 - e^{-\theta}}{\theta}\rbras*{ \abs*{\E*{F} - \theta} +  \Eabs*{\theta - \skp{DF}{-DL^{-1}F}_{L^2(\mu)}}}\\
			& \quad + \frac{1-e^{-\theta}}{\theta^2} \dsE \int\limits_{\XX} \abs[\big]{D_xF(D_xF-1)} \abs[\big]{D_xL^{-1}F} \mu(\id x).
	\end{align*}
\end{theorem}

The main idea of the proof is to take Mehler's formula and its application from \cite[Sections 3 and 4]{LastPeccatiSchulte2016} and adapt this technique for the bound given by Theorem \ref{thm:poisson-limit:malliavin-bound}. 

\begin{proof}[Proof of Theorem \ref{thm:poisson-limit}]
	Using the Cauchy-Schwarz inequality we can bound the first term by
	
	\begin{align*}
		\Eabs*{\theta - \skp*{DF}{-DL^{-1}F}_{L^2(\mu)}} & \leq \abs*{\V*{F} - \theta} + \Eabs*{\V*{F} - \skp*{DF}{-DL^{-1}F}_{L^2(\mu)}}\\
		& \leq \abs*{\V*{F} - \theta} + \sqrt{\E*{\rbras*{\V*{F} - \skp*{DF}{-DL^{-1}F}_{L^2(\mu)}}^2}},
	\end{align*}
	
	and apply Theorem \ref{thm:malliavin:var} to derive
	\begin{align*}
		\E*{\rbras*{\V*{F} - \skp*{DF}{-DL^{-1}F}_{L^2(\mu)}}^2} \leq 4\gamma_1(F) + \gamma_2(F),
	\end{align*}
	which yields the first part of our bound
	\begin{align*}
		\Eabs*{\theta - \skp*{DF}{-DL^{-1}F}_{L^2(\mu)}} \leq \abs*{\V*{F} - \theta} + 2 \sqrt{\gamma_1(F)} + \sqrt{\gamma_2(F)}.
	\end{align*}
	
	The second term can be bounded by using Fubini's theorem and H\"olders-inequality with parameters $p = q = 2$. Thus
	
	\begin{align*}
		& \dsE \int\limits_{\XX} \abs[\big]{D_xF(D_xF-1)} \abs[\big]{D_xL^{-1}F} \mu(\id x) \\
		= & \int\limits_{\XX} \E*{\abs[\big]{D_xF(D_xF-1)} \abs[\big]{D_xL^{-1}F}} \mu(\id x)\\
		\leq & \int\limits_{\XX} \rbras*{\Eabs[\big]{D_xF(D_xF-1)}^2}^\frac{1}{2} \rbras*{\Eabs[\big]{D_xL^{-1}F}^2}^\frac{1}{2} \mu(\id x),
	\end{align*}
	
	which can be bounded using Lemma \ref{lem:malliavin:moments} by
	
	\begin{align*}
		\dsE \int\limits_{\XX} \abs[\big]{D_xF(D_xF-1)} \abs[\big]{D_xL^{-1}F} \mu(\id x) \leq \int\limits_{\XX} \rbras*{\Eabs[\big]{D_xF(D_xF-1)}^2}^\frac{1}{2} \rbras*{\Eabs[\big]{D_xF}^2}^\frac{1}{2} \mu(\id x),
	\end{align*}
	
	yielding the second part of our bound
	
	\begin{align*}
		\dsE \int\limits_{\XX} \abs[\big]{D_xF(D_xF-1)} \leq \gamma_{3,P}(F),
	\end{align*}
	completing the proof of Theorem \ref{thm:poisson-limit}.
\end{proof}

\section{\texorpdfstring{A general bound for second-order $U$-statistics}{A general bound for second-order U-statistics}}
\label{sec:general-bound}

In this section, we adapt the general bound for the normal approximation of second-order $U$-statistics, that was provided in \cite[Section 3]{GrygierekThaele2016} to the Poisson case, showing that some of the previous results therein can be reused.
Let $F_d$ denote a second-order $U$-statistics in the sense of \cite{ReitznerSchulte2013} based on a Poisson point process in $\RR^d$ having intensity measure $\mu$.
Formally we define 
\begin{align*}
	F_d := \frac{1}{2} \sum\limits_{(y_1,y_2) \in \eta^2_{\neq}} h(y_1,y_2)
\end{align*}
and assume that $h:\RR^d \times \RR^d \rightarrow \cbras*{0,1}$ is a symmetric measurable function, which we allow to depend on the space dimension $d$. Furthermore, we assume that $\E*{F_d^2} < \infty$.
Finally we define the two parameter integrals
\begin{align*}
	A(x) & := \int\limits_{\RR^d} h(x,y) \mu(\id y), \quad x \in \RR^d,\\
	B(x_1,x_2) & := \int\limits_{\RR^d} h(x_1,y)h(x_2,y) \mu(\id y), \quad x_1,x_2 \in \RR^d,
\end{align*}
cf. \cite[Section 3]{GrygierekThaele2016}, where we already omit the exponents of $h: \RR^d \times \RR^d \rightarrow \cbras*{0,1}$.

Following \cite[Section 3]{GrygierekThaele2016}, by Mecke's formula we have that
\begin{align}\label{eq:u-stat:expectation}
	\E*{F_d} = \frac{1}{2} \int\limits_{\RR^d}\int\limits_{\RR^d} h(x_1,x_2) \mu(\id x_1) \mu(\id x_2)
\end{align}
and
\begin{align}\label{eq:u-stat:variance}
	\begin{split}
	\V*{F_d} & = \int\limits_{\RR^d} \rbras*{\int\limits_{\RR^d} h(x_1,x_2) \mu(\id x_2)}^2 \mu(\id x_1) + \frac{1}{2}\int\limits_{\RR^d}\int\limits_{\RR^d}h(x_1,x_2) \mu(\id x_2) \mu(\id x_1)\\
		& = \E*{F} + \int\limits_{\RR^d} \rbras*{\int\limits_{\RR^d} h(x_1,x_2) \mu(\id x_2)}^2 \mu(\id x_1).
	\end{split}
\end{align}

Next, we compute the expectations occurring at the right-hand side of Theorem \ref{thm:poisson-limit} to prepare the bounds for the three terms $\gamma_1(F_d)$, $\gamma_2(F_d)$, and $\gamma_{3,P}(F_d)$.

\begin{lemma}\label{lemma:DxF}
	Let $x,x_1,x_2 \in \RR^d$.
	Then
	\begin{enumerate}[(a)]
		\item $\E{(D_xF_d)^2} = A(x)^2 + A(x)$, \label{lemma:DxF:2}
		\item $\E{\abs*{D_xF_d}^3} = A(x)^3 + 3A(x)^2 + A(x)$, \label{lemma:DxF:3}
		\item $\E{(D_xF_d)^4} = P(x)$, with \label{lemma:DxF:4}
		\begin{align*}
			P(x) := A(x)^4 + 6A(x)^3 + 7A(x)^2 + A(x),
		\end{align*}
		\item $\E{(D_xF_d(D_xF_d-1))^2} = Q(x)$, with \label{lemma:DxF:2m}
		\begin{align*}
			Q(x) := A(x)^4 + 4A(x)^3 + 2A(x)^2,
		\end{align*}
		\item $\E{(D_{x_1,x_2}F_d)^4} = h(x_1,x_2)$. \label{lemma:DxyF:4}
	\end{enumerate}
\end{lemma}

\begin{proof}
	Assertions (\ref*{lemma:DxF:3}), (\ref*{lemma:DxF:4}) and (\ref*{lemma:DxyF:4}) are following directly from \cite[Lemma 3]{GrygierekThaele2016} using $h^2 = h$. Additionally the proof of (\ref*{lemma:DxF:2}) is similar to the proof of (\ref*{lemma:DxF:3}) writing
	\begin{align*}
		\E{(D_xF)^2} = \dsE\sum\limits_{(y_1, y_2) \in \eta^2} h(y_1,x) h(y_2,x).
	\end{align*}
	To prove \ref*{lemma:DxF:2m}) we write
	\begin{align*}
		\E{(D_xF_d(D_xF_d-1))^2} = \E{(D_xF)^4} - 2\E{(D_xF)^3} + \E{(D_xF)^2}
	\end{align*}
	and obtain $Q(x)$ using (\ref*{lemma:DxF:2}) and (\ref*{lemma:DxF:3}) combined with $D_xF \geq 0$ and (\ref*{lemma:DxF:4}).
\end{proof}

We shall now provide the announced expressions for the terms $\gamma_1(F_d)$, $\gamma_2(F_d)$ and $\gamma_{3,P}(F_d)$.

\begin{lemma}
	We have that
	\begin{align*}
		\gamma_1(F_d) & = \int\limits_{\RR^d}\int\limits_{\RR^d} B(x_1,x_2) \rbras[\Big]{P(x_1)P(x_2)}^\frac{1}{4} \mu(\id x_1) \mu(\id x_2),\\
		\gamma_2(F_d) & = \int\limits_{\RR^d}\int\limits_{\RR^d} B(x_1,x_2) \mu(\id x_1) \mu(\id x_2),\\
		\gamma_{3,P}(F_d) & = \int\limits_{\RR^d} \rbras[\Big]{Q(x)\rbras*{A(x)^2 + A(x)}}^\frac{1}{2} \mu(\id x)\\
			& = \int\limits_{\RR^d} \rbras[\Big]{A(x)^6 + 5A(x)^5 + 6A(x)^4 + 2A(x)^3}^\frac{1}{2} \mu(\id x).
	\end{align*}
\end{lemma}

\begin{proof}
	The expression for $\gamma_1(F_d)$ and $\gamma_2(F_d)$ are following similar to \cite[Lemma 4]{GrygierekThaele2016} by replacing the standardized $U$-statistics $\widetilde{F_d}$ with the non-standardized $F_d$. 
	Using Lemma \ref{lemma:DxF} (\ref*{lemma:DxF:2}) and (\ref*{lemma:DxF:2m}) we have that
	\begin{align*}
		\gamma_{3,P}(F) & = \int\limits_{\RR^d} \rbras*{\Eabs*{D_xF(D_xF-1)}^2}^\frac{1}{2} \rbras*{\Eabs*{D_xF}^2}^\frac{1}{2} \mu(\id x)\\
			& = \int\limits_{\RR^d} \rbras*{Q(x)}^\frac{1}{2} \rbras*{A(x)^2+A(x)}^\frac{1}{2} \mu(\id x)\\
			& = \int\limits_{\RR^d} \rbras[\Big]{Q(x)\rbras*{A(x)^2 + A(x)}}^\frac{1}{2} \mu(\id x),
	\end{align*}
	and the proof is complete.
\end{proof}

Now we can combine these expressions established so far to reformulate Theorem \ref{thm:poisson-limit} for our second-order $U$-statistic $F_d$.

\begin{proposition}\label{prop:general-bound:u-stat}
	Let $\eta$ be a Poisson point process on $\XX$ with $\sigma$-finite non-atomic intensity measure $\mu$ and let $F_d := \frac{1}{2} \sum_{(y_1,y_2) \in \eta^2_{\neq}} h(y_1,y_2)$ be a second-order $U$-statistic with symmetric kernel $h:\RR^d \times \RR^d \rightarrow \cbras*{0,1}$. Suppose that $\dsE \int_{\RR^d} (D_xF_d)^2 \mu(\id x) < \infty$. Defining
	\begin{align*}
		\gamma_1(F_d) & = \int\limits_{\RR^d}\int\limits_{\RR^d} B(x_1,x_2) \rbras[\Big]{P(x_1)P(x_2)}^\frac{1}{4} \mu(\id x_1) \mu(\id x_2),\\
		\gamma_2(F_d) & = \int\limits_{\RR^d}\int\limits_{\RR^d} B(x_1,x_2) \mu(\id x_1)\mu(\id x_2),\\
		\gamma_{3,P}(F_d) & = \int\limits_{\RR^d} \rbras[\Big]{Q(x)\rbras*{A(x)^2 + A(x)}}^\frac{1}{2} \mu(\id x),
	\end{align*}
	one has that
	\begin{align*}
		\dist_{TV}\rbras*{F,\cP(\theta)} \leq \frac{1-e^{-\theta}}{\theta} \rbras*{2\sqrt{\gamma_1(F_d)} + \sqrt{\gamma_2(F_d)} + \frac{\gamma_{3,P}(F_d)}{\theta} + \abs*{\E*{F} - \theta} + \abs*{\V*{F} - \theta}},
	\end{align*}
	where $\cP(\theta)$ is a Poisson distributed random variable with parameter $\theta > 0$.
\end{proposition}

\section{Proof of Theorem \ref{thm:edge-count-poisson-limit}}\label{sec:proof:edge-count-poisson-limit}

Let us recall that $\eta_d$ denotes a stationary Poisson point process on $\RR^d$ with intensity $\lambda_d$ given by \eqref{eq:phase:2}. 
We denote by $\mu$ the intensity measure of $\eta_d$, that is, $\mu$ is $\lambda_d$ times the Lebesgue measure on $\RR^d$. 
Moreover, from now on we will assume without loss of generality that all the random variables $(\cE_d)_{d \geq 2}$ are defined on a common probability space $(\Omega,\sF,\dsP)$.

It easy to see, that the edge counting statistic $\cE_d$ is a second-order $U$-statistic with measurable, symmetric and $d$-dependent kernel $h:\RR^d \times \RR^d \rightarrow \cbras*{0,1}$, given by
\begin{align}\label{eq:edge-count:kernel}
	h(x,y) := \1\cbras*{\norm{x-y} \leq \delta_d, \frac{x+y}{2} \in \BB^d}.
\end{align}
To derive Theorem \ref{thm:edge-count-poisson-limit} we apply the Poisson approximation bound derived in Proposition \ref{prop:general-bound:u-stat} using the bounds on the parameter integrals and the expectation and variance of $\cE_d$ given by \cite[eq. 15, Lemma 6, Lemma 7]{GrygierekThaele2016}.
We have
\begin{align}\label{eq:edge-count:expectation}
	\E*{\cE_d} = \frac{1}{2} \int\limits_{\RR^d} A(x) \lambda(\id x) = \frac{1}{2} \kappa_d^2 \lambda_d^2 \delta_d^{2d},
\end{align}
and
\begin{align}\label{eq:edge-count:variance}
	\frac{1}{2} \kappa_d^2 \lambda_d^2 \delta_d^d + \rbras*{1-\frac{\delta_d}{2}}^d \kappa_d^3 \lambda_d^3 \delta_d^{2d} \leq \V*{\cE_d} \leq \frac{1}{2} \kappa_d^2 \lambda_d^2 \delta_d^d + \rbras*{1+\frac{\delta_d}{2}}^d \kappa_d^3 \lambda_d^3 \delta_d^{2d}.
\end{align}
\begin{lemma}
	Let $h:\RR^d \times \RR^d \rightarrow \cbras*{0,1}$ be the function given by \eqref{eq:edge-count:kernel}. Then for all $x \in \RR^d$ it holds that
	\begin{align}\label{eq:edge-count:integral-bound}
		\1\cbras*{x \in \BB^d_{1-\frac{\delta_d}{2}}(0)} \kappa_d \lambda_d \delta_d^d \leq A(x) \leq \1\cbras*{x \in \BB^d_{1+\frac{\delta_d}{2}}(0)} \kappa_d \lambda_d \delta_d^d.
	\end{align}
\end{lemma}

\begin{remark}[cf. {\cite[Remark 8]{GrygierekThaele2016}}]\label{rem:choice-of-delta}
	Our particular choice $\delta_d = \frac{1}{d}$ ensures that we can find absolute constants $\bC_1, \bC_2 \in (0,\infty)$ and $\bD \in \NNN$ such that
	\begin{align}\label{eq:cond:delta:low}
		0 < \bC_1  \leq (1 - \frac{\delta_d}{2})^d,
	\end{align}
	and
	\begin{align}\label{eq:cond:delta:high}
		(1+\frac{\delta_d}{2})^d \leq \bC_2 < \infty,
	\end{align}
	for all $d \geq \bD$. The existence of such constants is important to derive the final bounds on the right hand side of our main result and implies restrictions to more general choices of $\delta_d$, see the proof of Lemma \ref{lem:edge-count:gamma-bounds}.
	If one is only interested in the Poisson limit, the first condition \eqref{eq:cond:delta:low} can be omitted, since it is only involved in the lower variance bound used in the Gaussian approximations, see \cite[Lemma 11 and eq. 20]{GrygierekThaele2016}.
\end{remark}

In the next step, we check the integrability condition in Proposition \ref{prop:general-bound:u-stat}. Note that this condition determines the limiting distribution, yielding the Gaussian limit if \eqref{eq:phase:1} holds resp. the Poisson limit if \eqref{eq:phase:2} holds:

\begin{lemma}
	If \eqref{eq:phase:1} holds, we have
	\begin{align*}
		\frac{1}{\V*{\cE_d}} \dsE \int\limits_{\RR^d}(D_x\cE_d)^2 \mu(\id x) = 1 + \frac{\E*{\cE_d}}{\V*{\cE_d}} < \infty,
	\end{align*}
	and \cite[Theorem 1]{GrygierekThaele2016} yields the Gaussian limit for the standardized edge counting statistics $(\V*{\cE_d})^{-\frac{1}{2}}(\cE_d - \E*{\cE_d})$.
	
	If \eqref{eq:phase:2} holds, we have
	\begin{align*}
		\dsE \int\limits_{\RR^d}(D_x\cE_d)^2 \mu(\id x) = \V*{\cE_d} + \E*{\cE_d} < \infty,
	\end{align*}
	thus we can apply the Poisson approximation given by Proposition \ref{prop:general-bound:u-stat}.
\end{lemma}

\begin{proof}
	The first claim was already shown by \cite[Lemma 9]{GrygierekThaele2016}.
	For the second claim, note that
	\begin{align*}
		\dsE \int\limits_{\RR^d}(D_x\cE_d)^2 \mu(\id x) = \V*{\cE_d} + \E*{\cE_d}.
	\end{align*}
	Using \eqref{eq:edge-count:expectation} and \eqref{eq:edge-count:variance} we obtain
	\begin{align*}
		\V*{\cE_d} + \E*{\cE_d} \leq \kappa_d^2 \lambda_d^2 \delta_d^d + \rbras*{1+\frac{\delta_d}{2}}^d \kappa_d^3 \lambda_d^3 \delta_d^{2d}.
	\end{align*}
	Assumption \eqref{eq:phase:2}, $\E*{\cE_d} = \frac{1}{2}\kappa_d^2 \lambda_d^2 \delta_d^d \rightarrow \theta$, implies that $\kappa_d^3 \lambda_d^3 \delta_d^{2d} \rightarrow 0$. 
	The choice of $\delta_d$ according to Remark \ref{rem:choice-of-delta} ensures that $\rbras*{1+\frac{\delta_d}{2}}^d$ can be bounded. 
	Thus $\V*{\cE_d} \rightarrow \theta$ and further $\V*{\cE_d} + \E*{\cE_d} < \infty$.
\end{proof}

Now, we will use the bounds for the parameter integral $A(x)$ to derive an upper bound for the three terms appearing in Proposition \ref{prop:general-bound:u-stat}.

\begin{lemma}\label{lem:edge-count:gamma-bounds}
	There are absolute constants $\bC_1, \bC_2, \bC_3 \in (0,\infty)$ and $\bD \in \NNN$ such that
	\begin{align*}
		\gamma_1(\cE_d) & \leq \bC_1 (\kappa_d\lambda_d\delta_d^d)^\frac{3}{2},\\
		\gamma_2(\cE_d) & \leq \bC_2 (\kappa_d\lambda_d\delta_d^d),\\
		\gamma_{3,P}(\cE_d) & \leq \bC_3 (\kappa_d\lambda_d\delta_d^d)^\frac{1}{2},
	\end{align*}
	for all $d \geq \bD$.
\end{lemma}

\begin{proof}
	Applying \eqref{eq:edge-count:integral-bound} to the definition of $P(x)$ in Lemma \ref{lemma:DxF} we see that
	\begin{align*}
		P(x) & \leq \1\cbras*{x \in \BB^d_{1+\frac{\delta_d}{2}}(0)} \sbras*{(\kappa_d\lambda_d\delta_d^d)^4+6(\kappa_d\lambda_d\delta_d^d)^3+7(\kappa_d\lambda_d\delta_d^d)^2+(\kappa_d\lambda_d\delta_d^d)^1}\\
			& \leq (\kappa_d\lambda_d\delta_d^d)^4+6(\kappa_d\lambda_d\delta_d^d)^3+7(\kappa_d\lambda_d\delta_d^d)^2+(\kappa_d\lambda_d\delta_d^d)^1.
	\end{align*}
	Therefore, it follows that
	\begin{align*}
		\gamma_1(\cE_d) & \leq \sbras*{(\kappa_d\lambda_d\delta_d^d)^4+6(\kappa_d\lambda_d\delta_d^d)^3+7(\kappa_d\lambda_d\delta_d^d)^2+(\kappa_d\lambda_d\delta_d^d)^1}^\frac{1}{2}\\
		& \quad \times \int\limits_{\RR^d}\int\limits_{\RR^d} B(x_1,x_2) \mu(\id x_1)\mu(\id x_2).
	\end{align*}
	We now use Fubini's theorem to re-write the double integral. Together with \eqref{eq:edge-count:integral-bound} this implies
	\begin{align*}
		\gamma_1(\cE_d) & \leq \sbras*{(\kappa_d\lambda_d\delta_d^d)^4+6(\kappa_d\lambda_d\delta_d^d)^3+7(\kappa_d\lambda_d\delta_d^d)^2+(\kappa_d\lambda_d\delta_d^d)^1}^\frac{1}{2} \int\limits_{\RR^d} A(y)^2 \mu(\id y)\\
			& \leq \kappa_d^3 \lambda_d^3 \delta_d^{2d} \rbras*{1+\frac{\delta_d}{2}}^d \sbras*{(\kappa_d\lambda_d\delta_d^d)^4+6(\kappa_d\lambda_d\delta_d^d)^3+7(\kappa_d\lambda_d\delta_d^d)^2+(\kappa_d\lambda_d\delta_d^d)^1}^\frac{1}{2}.
	\end{align*}
	Note that \eqref{eq:phase:2}, $\E*{\cE_d} = \frac{1}{2}\kappa_d^2 \lambda_d^2 \delta_d^d \rightarrow \theta$, implies $\kappa_d \lambda_d \delta_d^d \rightarrow 0$, thus the speed of convergence is dominated by the term with the lowest exponent.
	Additionally $\kappa_d^2\lambda_d^2\delta_d^d \rightarrow \theta$ and Remark \ref{rem:choice-of-delta} imply, that we can bound $(1+\frac{\delta_d}{2})^d$ and $\kappa_d^2\lambda_d^2\delta_d^d$ by absolute constants for $d$ sufficiently large.
	Thus there are absolute constants $\tilde{\bC}_1, \hat{\bC}_1, \bC_1 \in (0,\infty)$ and $\bD_1 \in \NNN$ such that
	\begin{align*}
		\gamma_1(\cE_d) & \leq \tilde{\bC}_1 \kappa_d^3 \lambda_d^3 \delta_d^{2d} \rbras*{1+\frac{\delta_d}{2}}^d (\kappa_d \lambda_d \delta_d^d)^\frac{1}{2}\\
			& \leq \hat{\bC}_1 (\kappa_d^2\lambda_d^2\delta_d^d)\rbras*{1+\frac{\delta_d}{2}}^d (\kappa_d \lambda_d \delta_d^d)^\frac{3}{2}\\
			& \leq \bC_1 (\kappa_d\lambda_d\delta_d^d)^\frac{3}{2},
	\end{align*}
	for all $d \geq \bD_1$.
	Using \eqref{eq:edge-count:integral-bound} we obtain in a similar way that
	\begin{align*}
		\gamma_2(\cE_d) \leq \int\limits_{\RR^d} A(y)^2 \mu(\id y) \leq \kappa_d^3 \lambda_d^3 \delta_d^{2d} \rbras*{1+\frac{\delta_d}{2}}^d \leq \bC_2 \kappa_d \lambda_d \delta_d^d,
	\end{align*}
	for all $d \geq \bD_2$, where $\bC_2 \in (0,\infty)$ and $\bD_2 \in \NNN$ are absolute constants.
	Applying \eqref{eq:edge-count:integral-bound} to the definition of $Q(x)$ in Lemma \ref{lemma:DxF} we see that
	\begin{align*}
		Q(x) & \leq \1\cbras*{x \in \BB^d_{1+\frac{\delta_d}{2}}(0)}\sbras*{(\kappa_d\lambda_d\delta^d)^4 + 4(\kappa_d\lambda_d\delta^d)^3 + 2(\kappa_d\lambda_d\delta^d)^2}\\
			& \leq (\kappa_d\lambda_d\delta^d)^4 + 4(\kappa_d\lambda_d\delta^d)^3 + 2(\kappa_d\lambda_d\delta^d)^2
	\end{align*}
	and
	\begin{align*}
		A(x)^2 + A(x) \leq \1\cbras*{x \in \BB^d_{1+\frac{\delta_d}{2}}(0)}\sbras*{(\kappa_d\lambda_d\delta^d)^2 + (\kappa_d\lambda_d\delta^d)}.
	\end{align*}
	Therefore, it follows that
	\begin{align*}
		\gamma_{3,P}(\cE_d) & \leq \sbras*{(\kappa_d\lambda_d\delta^d)^4 + 4(\kappa_d\lambda_d\delta^d)^3 + 2(\kappa_d\lambda_d\delta^d)^2}^\frac{1}{2}\int\limits_{\RR^d} \rbras*{A(x)^2 + A(x)}^\frac{1}{2} \mu(\id x)\\
			& \leq \sbras*{(\kappa_d\lambda_d\delta^d)^4 + 4(\kappa_d\lambda_d\delta^d)^3 + 2(\kappa_d\lambda_d\delta^d)^2}^\frac{1}{2} \sbras*{(\kappa_d\lambda_d\delta^d)^2 + (\kappa_d\lambda_d\delta^d)}^\frac{1}{2}\\
			& \quad \times \rbras*{1+\frac{\delta_d}{2}}^d \kappa_d \lambda_d\\
			& \leq \tilde{\bC}_3 (\kappa_d\lambda_d\delta^d)^\frac{3}{2} \kappa_d \lambda_d
			\leq \hat{\bC}_3 \sbras*{(\kappa_d^2\lambda_d^2\delta^d)^2 (\kappa_d\lambda_d\delta_d^d)}^\frac{1}{2}
			\leq \bC_3 (\kappa_d\lambda_d\delta^d)^\frac{1}{2}.
	\end{align*}
	for all $d \geq \bD_3$, where $\tilde{\bC}_3, \hat{\bC}_3, \bC_3 \in (0,\infty)$ and $\bD_3 \in \NNN$ are absolute constants.
	Setting $D := \max\cbras*{\bD_1,\bD_2,\bD_3}$ completes the proof.
\end{proof}

After these preparations, we can now present the proof of our second main result.
\begin{proof}[Proof of Theorem \ref{thm:edge-count-poisson-limit}]
	We use Proposition \ref{prop:general-bound:u-stat} and the results of the last lemma.
	Assuming \eqref{eq:phase:2} we find absolute constants $\bC_1, \bC_2, \bC_3, \bC_4, \bC_5 \in (0,\infty)$ and $\bD \in \NNN$ such that
	\begin{align*}
		\dist_{TV}\rbras*{\cE,\cP(\theta)} 
			& \leq \frac{1-e^{-\theta}}{\theta} 
			\left(
				2\sqrt{\bC_1}(\kappa_d\lambda_d\delta_d^d)^\frac{3}{4} + \sqrt{\bC_2}(\kappa_d\lambda_d\delta_d^d)^\frac{1}{2} + \frac{\bC_3}{\theta}(\kappa_d\lambda_d\delta_d^d)^\frac{1}{2} \right. \\
			& \left. + \bC_4\abs*{\frac{1}{2}\kappa_d^2 \lambda_d^2 \delta_d^d - \theta} + \bC_5\abs*{\frac{1}{2}\kappa_d^d\lambda_d^2 \delta^d + \kappa_d^3 \lambda_d^3 \delta_d^{2d}- \theta }\right)\\
			& \leq \frac{1-e^{-\theta}}{\theta} 
			\left( 2\sqrt{\bC_1}(\kappa_d\lambda_d\delta_d^d)^\frac{3}{4} + (\sqrt{\bC_2} + \frac{\bC_3}{\theta})(\kappa_d\lambda_d\delta_d^d)^\frac{1}{2} \right. \\
			& \left. + (\bC_4+\bC_5) \abs*{\frac{1}{2} \kappa_d^2 \lambda_d^2 \delta_d^d - \theta} + \bC_5 (\kappa_d^3 \lambda_d^3 \delta_d^{2d}) \right).
	\end{align*}
	holds for all $d \geq \bD$.
	Since $\kappa_d \lambda_d \delta_d^d \rightarrow 0$ the first and the last term are converging faster to zero than the second term, thus we can find absolute constants $\widetilde{\bC}_1, \widetilde{\bC}_2 \in (0,\infty)$ and $\widetilde{\bD}$ such that
	\begin{align*}
		\dist_{TV}\rbras*{\cE_d,\cP(\theta)} 
			& \leq \frac{1-e^{-\theta}}{\theta} \left( \widetilde{\bC}_1 (\kappa_d\lambda_d\delta_d)^\frac{1}{2} + \widetilde{\bC}_2 \abs*{\frac{1}{2} \kappa_d^2 \lambda_d^2 \delta_d^d - \theta }\right)
	\end{align*}
	holds for all $d \geq \widetilde{\bD}$.
	Using our assumption \eqref{eq:phase:2} it follows that $\dist_{TV}\rbras*{\cE_d,\cP(\theta)} \rightarrow 0$ and hence $\cE \os{D}{\longrightarrow} \cP(\theta)$ as $d \rightarrow \infty$.
	This completes the proof of Theorem \ref{thm:edge-count-poisson-limit}.
\end{proof}

\bigskip

%\AtNextBibliography{\small}
\printbibliography

\listoffixmes

\end{document}